\documentclass[11pt]{amsart}
\usepackage{amsmath,amssymb,amscd,verbatim,color}

\newtheorem{Theorem}{Theorem}[section]
\newtheorem{Lemma}{Lemma}[section]
\newtheorem{Proposition}{Proposition}[section]
\newtheorem{Corollary}{Corollary}[section]

\numberwithin{equation}{section}

\def\bc{\begin{center}}
\def\ec{\end{center}}

\def\a1{(a_1, a_2, \cdots, a_n)}

\allowdisplaybreaks

\begin{document}

\title[Stable sets and mean Li-Yorke chaos in positive entropy actions]
{Stable sets and mean Li-Yorke chaos in positive entropy
actions of bi-orderable amenable groups}

\author[W. Huang]{Wen Huang}
\address{W. Huang: Wu Wen-Tsun Key Laboratory of Mathematics, USTC, Chinese Academy of Sciences, Department of Mathematics, University of Science and Technology of China, Hefei, Anhui, 230026, P.R.China}
\email{wenh@mail.ustc.edu.cn}
\author[L. Jin]{Lei Jin}
\address{L. Jin: Department of Mathematics, University of Science and Technology of China, Hefei, Anhui, 230026, P.R.China}
\email{jinleim@mail.ustc.edu.cn}

\thanks {Authors are supported by NNSF of China 11225105, 11371339 and 11431012.}



\begin{abstract} It is proved that positive entropy implies mean Li-Yorke chaos for a $G$-system,
where $G$ is a countable, infinite, discrete, bi-orderable amenable group. Examples are given for the
cases of integer lattice groups and groups of integer unipotent upper triangular matrices.
\end{abstract}

\maketitle

\section{Introduction}
Given a $G$-system $(X,G)$, the chaotic behavior of $(X,G)$ reflects the complexity of the system,
which is one of the central topics in the study of dynamical systems.
The most popular notions used to describe the chaotic behaviors are positive entropy, weak mixing,
Devaney's chaos, Li-Yorke chaos and distributional chaos. Thus, the relationship among these notions
attracts a lot of attention.

For $\mathbb{Z}$-actions, it was shown by Huang and Ye in \cite{AAHY} that Devaney's chaos implies Li-Yorke chaos, and it was
proved by Blanchard, Glasner, Kolyada and Maass (see \cite{BGKM}) that positive entropy also implies Li-Yorke chaos.
It is clear that weak mixing and distributional chaos imply Li-Yorke chaos too. Distributional chaos can be divided further
to DC1, DC2 and DC3. It is easy to see that positive entropy does not imply DC1. By a nice observation, Downarowicz showed in \cite{dnrwc}
that DC2 is equivalent to the so-called mean Li-Yorke chaos, and proved that positive entropy implies DC2. Using
more standard argument, Huang, Li and Ye \cite{HLY} gave another proof of the later result and obtained the relationship between
asymptotic, entropy and mean Li-Yorke tuples.

Now a natural question is whether the above results hold for general group actions?
Kerr and Li first showed for amenable groups (see \cite{KL1}), then for sofic groups (see \cite{KL2}) that positive entropy indeed implies
Li-Yorke chaos. In \cite{AAHXY}, the authors considered asymptotic pairs, stable sets and chaos in positive
entropy $G$-systems for certain countable discrete infinite left-orderable amenable groups $G$.

Motivated by the above ideas and results, we naturally ask and study whether
positive entropy also implies mean Li-Yorke chaos for a general $G$-system.
In this paper, we investigate stable sets and mean Li-Yorke chaos for $G$-systems with positive entropy,
and prove in Theorem \ref{thm1} that positive entropy indeed implies mean Li-Yorke chaos for general actions
of countable discrete infinite bi-orderable amenable groups.
To state our main result of this paper, we introduce some notions in the following.

Let $G$ be a group with the unit $e_G$.
$G$ is said to be {\it left-orderable} if there exists a strict total ordering $<$ on its elements which is left invariant,
that is, $g_1<g_2$ implies $g_0g_1<g_0g_2$ for all $g_0,g_1,g_2\in G$.
If $<$ is also invariant under right-multiplication, then we say that $G$ is {\it bi-orderable}.
It is well known that a group $G$ is left-orderable if and only if it is right-orderable if and only if it contains a subset
$\Phi$, called an {\it algebraic past} of $G$, such that
$\Phi\cdot \Phi\subset\Phi$, $\Phi\cap\Phi^{-1}=\emptyset$, and $\Phi\cup \Phi^{-1} \cup\{e_G\}=G$.
Moreover, $G$ is bi-orderable if and only if there exists such $\Phi$ which also satisfies $g\Phi g^{-1}\subset\Phi$ for all $g\in G$.
Indeed, when $(G,<)$ is bi-orderable, one can take $\Phi =\{g\in G: g<e_G\}$.
With respect to the algebraic past $\Phi$, we obtain the desired linear ordering on $G$ as follows: $g_1$ is
less than $g_2$ (write $g_1<_{\Phi} g_2$), if $g_2^{-1}g_1\in\Phi$, and also $g_1g_2^{-1}\in\Phi$.
The simplest bi-orderable groups are the torsion-free, abelian ones.
Torsion-free, nilpotent groups and non-abelian free groups are also bi-orderable.
A nontrivial left-orderable group must be torsion-free.
The reader can refer to, e.g. \cite{BR,KM,GOD},
for more details of the theory of orderable groups.

Next we introduce the following version of mean Li-Yorke chaos for a $G$-system.
To state this definition, we first let $G$ be a countable discrete infinite amenable group,
and $\mathcal{F}=\{F_n\}_{n=1}^\infty$ be a F{\o}lner sequence of $G$.
Let $(X,G)$ be a {\it $G$-system}, that is, $X$ is a compact metric space and the group $G$ continuously acts on $X$.
Denote by $\rho$ the metric on $X$. Given $\eta>0$,
a pair $(x,y)\in X\times X$ is called a
{\it mean Li-Yorke pair with modulus $\eta$} for $\mathcal{F}$, if
\begin{align*}
\liminf_{n\rightarrow\infty} \frac{1}{|F_n|} \sum_{g\in F_n} \rho(gx,gy)=0 \text{ and }\limsup_{n\rightarrow\infty}\frac{1}{|F_n|} \sum_{g\in F_n} \rho(gx,gy)\ge \eta.
\end{align*}
For such $\mathcal{F}$, denote by $MLY_\eta (X,G)$ the set of all mean Li-Yorke pairs with modulus $\eta$ in $(X,G)$.
A subset $K$ of $X$ is called a {\it mean Li-Yorke set (with modulus $\eta$)}
if every $(x,y)\in K\times K$ with $x\neq y$ is a mean Li-Yorke pair (with modulus $\eta$), i.e.,
$K\times K\setminus\Delta_X\subset MLY_\eta (X,G)$,
where $\Delta_X=\{(x,x): x\in X\}$.

For a given $G$-system $(X,G)$, let $S$ be an infinite subset of $G$.
In \cite{AAHXY}, the authors introduced $S$-asymptotic pairs and $S$-stable sets for the $G$-system as follows.
 A pair $(x,y)\in X\times X$ is said to be an {\it $S$-asymptotic pair} if for each $\epsilon>0$, there are only finitely many elements
$s\in S$ with $\rho(sx,sy)>\epsilon$. Denote by $Asy_S(X,G)$ the set of all $S$-asymptotic pairs of $(X,G)$.
For a point $x\in X$, the set $W_S(x,G)=\{y\in X: (x,y)\in Asy_S(X,G)\}$
 is called the {\it $S$-stable set} of $x$.

\medskip

We now begin to state our main theorem.
For convenience, we say that $(G, \Phi, S, \mathcal{F})$ is a {\it cba-group} if
$G$ is a countable discrete infinite bi-orderable amenable group with the algebraic past $\Phi$,
$S$ is an infinite subset of $G$ such that $\sharp\{s\in S:s<_\Phi g\}<\infty$ for each $g \in G$,
and $\mathcal{F}=\{F_n\}_{n=1}^\infty$ is a F{\o}lner sequence of $G$ with
$e_G\in F_1\subset F_2\subset\cdots\subset\bigcup\limits_{n=1}^\infty F_n\subset S$.

A natural question is whether there always exist $S$ and $\mathcal{F}$
such that $(G, \Phi, S, \mathcal{F})$ is a cba-group
for given $G$ which is a countable discrete infinite bi-orderable amenable group with the algebraic past $\Phi$.
As the answer, we have the following result which will be proved in Section 2.

\begin{Proposition}\label{fexist}
Let $G$ be a countable discrete infinite bi-orderable amenable group with the algebraic past $\Phi$,
then there exist  an infinite subset $S$ of $G$ and  a F{\o}lner sequence $\mathcal{F}$ of $G$ such that $(G, \Phi, S, \mathcal{F})$ is a cba-group.
\end{Proposition}

\medskip

Our main result of this paper is the following

\begin{Theorem}\label{thm1} Let $(G, \Phi, S, \mathcal{F})$ be a cba-group.
If $(X,G)$ is a $G$-system and $\mu$ is an ergodic $G$-invariant Borel probability measure on $X$ with positive entropy,
then there exists a positive number $\eta$ satisfying that for $\mu$-a.e. $x\in X$: there exists a Cantor set
$K_x\subset\overline{W_S(x,G)}$ such that $K_x$ is a mean Li-Yorke set with modulus $\eta$ for $\mathcal{F}$.
\end{Theorem}

\medskip

Note that the order given in the cba-group is important for us to construct the excellent partition in the proof of this main theorem.
Moreover, by the variational principle of entropy for $G$-systems (see e.g., \cite{OP,ST}), we obtain the following result from Theorem \ref{thm1}.

\begin{Corollary}
Suppose $G$ is a countable discrete infinite bi-orderable amenable group with the algebraic past $\Phi$,
 $\mathcal{F}=\{F_n\}_{n=1}^\infty$ is a F{\o}lner sequence of $G$ with $e_G\in F_1\subset F_2\subset\cdots$,
and $\sharp\{s\in \bigcup\limits_{n=1}^\infty F_n:s<_\Phi g\}<\infty$ for each $g \in G$.
If $(X,G)$ is a $G$-system with positive topological entropy,
then there exist $\eta>0$ and a Cantor set $K\subset X$ such that $K$ is a mean Li-Yorke set with modulus $\eta$ for $\mathcal{F}$.
\end{Corollary}

\medskip

To demonstrate applications of Theorem \ref{thm1}, we give below two examples of F\o lner sequences for
some special cases.
We first consider the integer lattice group. It is clear that $\{F_n=[0,n-1]^d: n\in \mathbb N\}$ is a F{\o}lner sequence of $\mathbb{Z}^d$, where $d\in\mathbb{N}$.
Note that $\mathbb{Z}^d$ can be regarded as a countable discrete infinite bi-orderable amenable group with
 an algebraic past
(see e.g., \cite{AAHXY}).
Put $\mathbb{Z}_+^d=\{(n_1,\cdots, n_d)\in\mathbb{Z}^d : n_i\geq 0, 1\leq i \leq d\}$, and apply Theorem \ref{thm1},
we immediately get the following result.

\begin{Theorem}
Let $(X,\mathbb{Z}^d)$ be a $\mathbb{Z}^d$-system for some $d\in\mathbb{N}$, and $\mu$ be an ergodic $\mathbb{Z}^d$-invariant Borel probability measure on $X$ with positive entropy. Then there exists $\eta>0$ satisfying that for $\mu$-a.e. $x\in X$, there exists a Cantor set $K_x\subset\overline{W_{\mathbb{Z}_+^d}(x,\mathbb{Z}^d)}$ such that $K_x$ is a mean Li-Yorke set with modulus $\eta$ for $\{[0,n-1]^d: n\in \mathbb N\}$.
\end{Theorem}

We remark that even for $d=1$, when the entropy of the system $(X,\mathbb{Z})$ is positive,
 we do not know whether there exist mean Li-Yorke pairs for the F\o lner sequence $\{F_n=[-n,n] : n\in \mathbb N\}$ in the system.
For related topics see \cite{OW,DL}. One may ask whether mean Li-Yorke chaos only depends on the F\o lner sequence assuming that all F\o lner sets
sit in the future. We have the following proposition which will be proved in the next section.

\begin{Proposition}\label{fnots}
Let $G$ be a countable discrete infinite bi-orderable amenable group with the algebraic past $\Phi$. Then the following are equivalent:
\begin{itemize}
\item[{\rm 1)}] there exists $g\in G$ such that $g\Phi^{-1}\cap\Phi$ is infinite;
\item[{\rm 2)}] there exist $g\in G$ and a F\o lner sequence $\mathcal{F}=\{F_n\}_{n=1}^\infty$ of $G$ with
$e_G\in F_1\subset F_2\subset\cdots\subset\bigcup\limits_{n=1}^\infty F_n\subset \Phi^{-1}\cup\{e_G\}$, such that
 $\sharp\{s\in \bigcup\limits_{n=1}^\infty F_n :s<_\Phi g\}=\infty$.
In particular, for such F\o lner sequence $\mathcal{F}$ of $G$, there does not exist an infinite subset $S$ of $G$
 satisfying that $(G, \Phi, S, \mathcal{F})$ is a cba-group.
\end{itemize}
\end{Proposition}

According to Proposition \ref{fnots}, if we suppose $G=\mathbb{Z}^2$ and
$$\Phi= \{(z_1,z_2)\in \mathbb{Z}^2: z_1+z_2<0\}\cup\{(z_1,z_2)\in \mathbb{Z}^2: z_1+z_2=0, z_1<0\},$$
then it is easy to check that the condition (1) of Proposition \ref{fnots} holds (e.g., by choosing $g=(-100,-100)\in\mathbb{Z}^2$),
thus, we know that there is a F\o lner sequence $\mathcal{F}=\{F_n\}_{n=1}^\infty$ of $G$ with
$e_G\in F_1\subset F_2\subset\cdots\subset\bigcup\limits_{n=1}^\infty F_n\subset \Phi^{-1}\cup\{e_G\}$
(i.e. all F\o lner sets sit in the future)
such that there does not exist an infinite subset $S$ of $G$ satisfying that $(G, \Phi, S, \mathcal{F})$ is a cba-group.

\medskip

Now we turn to our second example of the countable group of integer unipotent upper triangular matrices.
For given $d\in \mathbb{N}$,
let
$$
M({\bf a})=\left(
  \begin{array}{cccccccc}
    1 & a_{1,1} & a_{1,2} & a_{1,3}  &\ldots & a_{1,d-1} &a_{1,d} \\
    0 & 1       & a_{2,1} & a_{2,2}  &\ldots & a_{2,d-2} &a_{2,d-1}\\
    0 & 0       & 1       & a_{3,1}  &\ldots & a_{3,d-3} &a_{3,d-2}\\
   \vdots & \vdots & \vdots &  \vdots & \vdots & \vdots &\vdots\\
    0 & 0    &0       &0           & \ldots  & a_{d-1,1} &a_{d-1,2} \\
    0 & 0    &0       &0            & \ldots & 1 &a_{d,1} \\
    0 & 0    &0       &0            & \ldots & 0  & 1
  \end{array}
\right),
$$
 where ${\bf a}=(a_{i,k})_{1\le k\le d, 1\le i\le d-k+1}\in
\mathbb{Z}^{d(d+1)/2}$.
We note that
the indexation is nonstandard,
as the second index refers to the
off-diagonal, not the column.
 Let the group
\begin{equation}\label{udz}
U_{d+1}(\mathbb{Z})=\{ M({\bf a}):{\bf a}=(a_{i,k})_{1\le k\le d, 1\le i\le d-k+1}\in
\mathbb{Z}^{d(d+1)/2}\}.
\end{equation}
Then the group $U_{d+1}(\mathbb{Z})$ of integer
unipotent upper triangular matrices is in fact a $d$-step nilpotent
group.
Put $$\mathbf{S}=\{{\bf a}=(a_{i,k})_{1\le k\le d, 1\le i\le d-k+1}\in
\mathbb{Z}^{d(d+1)/2} : a_{d,1}^k \ge a_{i,k}\ge 0 \text{ for }1\le k\le d, 1\le i\le d-k+1\},$$ and
\begin{equation}\label{udzs}
S=\{ M({\bf a}) :  {\bf a}\in\mathbf{S} \}.
\end{equation}
For each $n\in \mathbb{N}$, take
\begin{equation}\label{udzfn}
F_n=\{ M({\bf a})\in S : a_{d,1} \le n \}.
\end{equation}
We will show in Lemma \ref{fphilners} that such $\{F_n\}_{n\ge 1}$ is a F\o lner sequence of $U_{d+1}(\mathbb{Z})$.
Then by Theorem \ref{thm1}, we obtain the following result.

\begin{Theorem}\label{thm2}
Define $U_{d+1}(\mathbb{Z})$, $S$ and $\{F_n\}_{n=1}^\infty$ as in \eqref{udz}, \eqref{udzs} and \eqref{udzfn} respectively.
If $(X,U_{d+1}(\mathbb{Z}))$  is a  $U_{d+1}(\mathbb{Z})$-system and $\mu$  is an ergodic $U_{d+1}(\mathbb{Z})$-invariant Borel probability measure on $X$ with positive entropy, then there exists  $\eta > 0$  satisfying that for $\mu$-a.e. $x\in X$, there exists a Cantor set $K_x\subset\overline{W_S(x,U_{d+1}(\mathbb{Z}))}$ such that $K_x$ is a mean Li-Yorke set with modulus $\eta$ for $\{F_n\}_{n=1}^\infty$.
\end{Theorem}

\medskip

This paper is organized as follows. In Section 2,
we firstly review some basic dynamical properties for $G$-systems and also state some useful results,
then prove Propositions \ref{fexist} and \ref{fnots}.
In Section 3, we give the proof of Theorem \ref{thm1}. Finally, we prove Theorem \ref{thm2} in Section 4.

\medskip

\noindent{\bf Acknowledgement:  } We would like to thank Professor Xiangdong Ye for useful comments and helpful suggestions.
This paper has now been accepted by the journal {\it Ergodic Theory and Dynamical Systems}.
We also thank the anonymous referee for his/her helpful suggestions concerning this paper.

\medskip

\section{Preliminaries}
We firstly review some basic notions and fundamental properties.
Recall that a countable discrete group $G$ is called {\it amenable}
if there exists a sequence of finite subsets $F_n\subset G$,
such that $\lim\limits_{n\rightarrow +\infty} \frac{|gF_n\Delta F_n|}{|F_n|}=0$
holds for every $g \in G$, and we say that such $\{F_n\}_{n=1}^\infty$ is a {\it F{\o}lner sequence} of $G$.

Let $G$ be a countable discrete infinite amenable group with the unit $e_G$.
By a {\it $G$-system} $(X, G)$ we
mean that $X$ is a compact metric space endowed with the metric $\rho$, and
$\Gamma: G\times  X\rightarrow X, (g,x)\mapsto gx$ is a continuous mapping
satisfying $\Gamma (e_G, x)= x$, $\Gamma (g_1, \Gamma (g_2, x))= \Gamma (g_1 g_2, x)$
for each $g_1, g_2\in G$ and $x\in X$.
In the following, we fix a $G$-system $(X,G)$.

Denote by $\mathcal{B}_X$ the collection of all Borel subset of $X$
and $\mathcal{M}(X)$ the set of all Borel probability measures on $X$.
For $\mu\in \mathcal{M}(X)$, denote by $\mathcal{B}_X^\mu$ the completion of $\mathcal{B}_X$ under $\mu$.
The {\it support} of $\mu$ is defined to be the set
$$
\text{supp}(\mu) = \{x\in X: \mu (U)>0 \text{ for every open neighborhood } U \text{ of } x\} .
$$
It is clear that $\mu\big(\text{supp}(\mu)\big)=1$. Moreover, if $\mu(A)=1$ for $A\subset X$, then $\overline{A}\supset\text{supp}(\mu)$.
A sub-$\sigma$-algebra $\mathcal{A}$ of $\mathcal{B}_X^\mu$ is said
to be {\it $G$-invariant} if $g\mathcal{A}=\mathcal{A}$ for any $g\in G$.
$\mu\in \mathcal{M}(X)$ is called {\it $G$-invariant} if
$\mu = g \mu:=\mu\circ g^{-1}$ for each $g\in G$, and is
called {\it ergodic} if it is $G$-invariant and $\mu\big(\bigcup_{g\in G} g A\big)= 0$ or $1$ for any $A\in \mathcal{B}_X$.
Denote by $\mathcal{M} (X, G)$ the set of all $G$-invariant elements in
$\mathcal{M} (X)$. The set of all ergodic $G$-invariant elements in $\mathcal{M} (X)$ is denoted by $\mathcal{M}^e(X,G)$.
Note that the amenability of $G$ ensures that
 $\mathcal{M}^e (X, G)\ne \emptyset$ and both
$\mathcal{M}(X)$ and $\mathcal{M}(X, G)$ are convex compact metric
spaces under the weak$^*$-topology.

A {\it cover} of $X$ is a finite family of subsets of
$X$ whose union is $X$.
Given two covers  $\mathcal{U}$, $\mathcal{V}$ of $X$, $\mathcal{U}$
is said to be  {\it finer} than $\mathcal{V}$ (write $\mathcal{U} \succeq \mathcal{V}$) if each element of
$\mathcal{U}$ is contained in some element of $\mathcal{V}$.  Define
$\mathcal{U}$ $\vee$ $\mathcal{V}$ $= \{U \cap V : U \in \mathcal{U}, V \in \mathcal{V}\}$.
A  {\it partition} of $X$ is a cover of $X$ whose elements are pairwise disjoint.
Denote by $\mathcal{P}_X$ (resp. $\mathcal{P}_X^b$) the set of all partitions
(resp. finite Borel partitions) of $X$.
Let $\mathcal{P}_X^\mu=\{ \alpha\in \mathcal{P}_X: \text{each element in }\alpha \text{ belongs to }\mathcal{B}_X^\mu\}$.
The partition $\alpha=\bigvee_{i\in I}\alpha_i$ is called a {\it measurable
partition} if $\{ \alpha_i\}_{i\in I}$ is a countable family
of finite Borel partitions of $X$. The sets $A\in \mathcal{B}_X^\mu$ which are unions of
atoms of this $\alpha$ form a sub-$\sigma$-algebra of
$\mathcal{B}_X^\mu$ which is denoted by
$\widehat{\alpha}$ (or $\alpha$ if there is no ambiguity).
In fact, every sub-$\sigma$-algebra of
$\mathcal{B}_X^\mu$ coincides with a $\sigma$-algebra constructed in
 the way above (mod $\mu$).

For $f\in L^1(X,\mathcal{B}_X,\mu)$ and a sub-$\sigma$-algebra $\mathcal{A}$ of $\mathcal{B}_X^\mu$,
 denote by $\mathbb{E}_\mu (f | \mathcal{A})$ the conditional expectation (see e.g., \cite[Chapter 5]{AAEW}) of
$f$ with respect to $\mathcal{A}$.
The following result is well-known (see e.g., \cite[Theorem 14.26]{Gbook}, \cite[Chapter 5.2]{AAEW}).

\begin{Theorem} (Martingale Theorem) Let $\{\mathcal{A}_n\}_{n\ge 1}$ be a decreasing sequence (resp. an increasing sequence) of sub-$\sigma$-algebras of $\mathcal{B}_X^\mu$
and let $\mathcal{A}= \bigcap\limits_{n\ge 1} \mathcal{A}_n$ (resp.
$\mathcal{A}= \bigvee\limits_{n\ge 1} \mathcal{A}_n$). Then for every $f
\in L^1(X,\mathcal{B}_X,\mu)$,
$\mathbb{E}_\mu(f | \mathcal{A}_n)\rightarrow \mathbb{E}_\mu(f | \mathcal{A})$ as $n\rightarrow\infty$
in $L^1(\mu)$ and also $\mu$-almost everywhere.
\end{Theorem}

We begin to introduce the measure-theoretic entropy.
Given $\mu\in \mathcal{M}(X,G)$,  $\alpha\in \mathcal{P}_X^\mu$,
and a sub-$\sigma$-algebra $\mathcal{A}$ of $\mathcal{B}_X^\mu$,
 define
$$
H_{\mu}
(\alpha | \mathcal{A})= \sum_{A\in \alpha} \int_X
 -\mathbb{E}_\mu (1_A| \mathcal{A}) \log \mathbb{E}_\mu (1_A| \mathcal{A}) d \mu.
$$
One standard fact is that
$H_{\mu} (\alpha | \mathcal{A})$ increases with respect to $\alpha$
and decreases with respect to $\mathcal{A}$. Set $\mathcal{N}=
\{\emptyset, X\}$, and define
$
H_\mu (\alpha)= H_\mu (\alpha| \mathcal{N})
$.
Note that any $\beta \in \mathcal{P}_X^\mu$ naturally
generates a sub-$\sigma$-algebra $\mathcal{F} (\beta)$ of
$\mathcal{B}_X^\mu$, we then define
$
H_{\mu}(\alpha|\beta)= H_\mu (\alpha| \mathcal{F} (\beta))
$.
 The {\it measure-theoretic entropy} of $\mu$ relative to $\alpha$ is
defined by
\begin{equation*}
h_{\mu}(G,\alpha)=\lim\limits_{n\rightarrow+\infty}\frac{1}{|F_n|}
H_\mu(\bigvee_{g\in F_n}g^{-1}\alpha),
\end{equation*}
where $\{F_n\}_{n=1}^\infty$ is a F{\o}lner sequence in the group $G$.
It was shown in \cite[Theorem 6.1]{LW}  that the limit exists and
is independent of F{\o}lner sequences.
The {\it measure-theoretic entropy} of $\mu$ is defined by
$h_{\mu}(G)=h_{\mu}(G,X)=\sup_{\alpha\in \mathcal{P}^b_X}
h_{\mu}(G,\alpha)$.
The {\it conditional entropy} of $\alpha$ with respect to a $G$-invariant sub
$\sigma$-algebra $\mathcal{A}$ of  $\mathcal{B}_X^\mu$ is
 defined by
$$h_\mu(G,\alpha|\mathcal{A}) = \lim \limits_{n\rightarrow+\infty}\frac{1}{|F_n|}
H_\mu(\bigvee_{g\in F_n}g^{-1}\alpha|\mathcal{A}),$$  where $\{F_n\}_{n=1}^\infty$ is
 a F{\o}lner sequence of $G$.
By standard arguments,  one can similarly get from   \cite[Theorem 6.1]{LW}    that this limit exists and is also
 independent of F{\o}lner sequences.
The {\it conditional
entropy} of $\mu$ with respect to $\mathcal{A}$ is then defined by
$h_\mu(G|\mathcal{A})=\sup_{\alpha\in \mathcal{P}^b_X}
h_{\mu}(G,\alpha|\mathcal{A})$.

Next we turn to introducing conditional measures. For more details see for example \cite[Chapter 5]{AAEW}.
Let $\mathcal{F}$ be a sub-$\sigma$-algebra of
$\mathcal{B}_X^\mu$ and $\alpha$ be the measurable partition of
$X$ with $\widehat{\alpha}=\mathcal{F}$ (mod $\mu$). Then $\mu$ can be
disintegrated over $\mathcal{F}$ as
$$\mu=\int_X \mu_x d \mu(x),$$
where $\mu_x\in \mathcal{M}(X)$, $\mu_x(\alpha(x))=1$ for
$\mu$-a.e. $x\in X$, and if $\alpha(x)=\alpha(y)$ then $\mu_x=\mu_y$.
The disintegration is characterized by the
properties \eqref{meas1} and \eqref{meas3} below:
\begin{equation}\label{meas1}
\text{for every } f \in L^1(X,\mathcal{B}_X,\mu),  \text{  }  f \in L^1(X,\mathcal{B}_X,\mu_x)
\text{ for $\mu$-a.e. } x\in X,
\end{equation}
$$
 \text{  and the map   } x \mapsto \int_X  f(y) d\mu_x(y)  \text{  is in   }
L^1(X,\mathcal{F},\mu)  ;
$$
\begin{equation}\label{meas3}
\text{for every  } f\in L^1(X,\mathcal{B}_X,\mu),  \text{  }\text{  }  \mathbb{E}_{\mu}(f|\mathcal{F})(x)=\int_X
f d\mu_x  \text{  }  \text{   for $\mu$-a.e.  } x\in X.
\end{equation}

Let $\mathcal{F}$ be a sub-$\sigma$-algebra of
$\mathcal{B}_X^\mu$ and $\mu=\int_X \mu_x d \mu(x)$  be the
disintegration of $\mu$ over $\mathcal{F}$. The {\it conditional product}  of $\mu$ relative to
$\mathcal{F}$ is the  Borel probability measure
$\mu\times_\mathcal{F}\mu$ on $X\times X$ such that
\begin{equation*}
(\mu\times_\mathcal{F}\mu) (A\times B)=\int_X \mathbb{E}_\mu
(1_A|\mathcal{F})(x)\mathbb{E}_\mu
(1_B|\mathcal{F})(x)d\mu(x)=\int_X \mu_x(A)\mu_x(B)d\mu(x)
\end{equation*}
for all $A,B\in \mathcal{B}_X$.

\medskip

In the remaining part of this section, we prove Propositions \ref{fexist} and \ref{fnots} as stated in the introduction.

\begin{proof}[Proof of Proposition \ref{fexist}]
Since $G$ is a countable discrete infinite amenable group, there exists a F\o lner sequence $\{G_n\}_{n=1}^\infty$ of $G$.
We note that $|G_n|\rightarrow\infty$.

Write $\Phi^{-1}=\{g_n\}_{n=1}^\infty$. Since $G$ is bi-orderable, for each $n\in\mathbb{N}$,
we can choose $h_n\in G$ such that $\min_\Phi (G_nh_n)>_\Phi\max_\Phi\{g_i : 1\le i\le n\}$.
That is, for each element $\widetilde{g}\in G_nh_n$ and any $1\le i\le n$, we have $\widetilde{g}>_\Phi g_i$.

Put $H_n=G_nh_n\cup\{e_G\}$,
then $\{H_n\}_{n=1}^\infty$ is also a F\o lner sequence of $G$ since
$$\lim\limits_{n\rightarrow\infty}\frac{|gH_n\Delta H_n|}{|H_n|}=\lim\limits_{n\rightarrow\infty}\frac{|gG_n\Delta G_n|}{|G_n|}=0$$ for any $g\in G$.

Let $n_1=1$, and for every $i\ge2$, choose $n_i\in\mathbb{N}$ such that $|H_{n_i}|>2^i\sum_{j=1}^i|H_{n_{i-1}}|$.
Now for each $N\in\mathbb{N}$, take $F_N=\bigcup\limits_{i=1}^NH_{n_i}$, $S=\bigcup\limits_{N=1}^\infty F_N$, and $\mathcal{F}=\{F_n\}_{n=1}^\infty$.
Then we can check that $(G, \Phi, S, \mathcal{F})$ is a cba-group as follows.

Clearly, $S$ is infinite and $e_G\in F_1\subset F_2\subset\cdots\subset\bigcup\limits_{n=1}^\infty F_n= S$.
By the definition of $H_n$, we know that for each $g \in G$, $\{s\in S:s<_\Phi g\}$ is finite.
Finally, by the definition of $n_i$ and the fact that $\{H_n\}_{n=1}^\infty$ is a F\o lner sequence of $G$, we have
\begin{align*}
\lim\limits_{N\rightarrow\infty}\frac{|gF_N\Delta F_N|}{|F_N|}&\le
\lim\limits_{N\rightarrow\infty}\frac{|F_{N-1}|+|gF_{N-1}|+|gH_{n_N}\Delta H_{n_N}|}{|H_{n_N}|}\\
&\le\lim\limits_{N\rightarrow\infty}\frac{2\sum_{i=1}^{N-1}|H_{n_i}|+|gH_{n_N}\Delta H_{n_N}|}{|H_{n_N}|}\\
&=0
\end{align*}
for every $g\in G$, which implies that
$\{F_n\}_{n=1}^\infty$ is a F{\o}lner sequence of $G$.
\end{proof}

\medskip

\begin{proof}[Proof of Proposition \ref{fnots}]
(1)$\Rightarrow$(2): Suppose $g\in G$ is such that $|g\Phi^{-1}\cap\Phi|=\infty$.
Firstly, according to the proof of Proposition \ref{fexist}, we can choose a F\o lner sequence $\{E_n\}_{n=1}^\infty$ of $G$
with $|E_n|\rightarrow\infty$ such that
$e_G\in E_1\subset E_2\subset\cdots\subset\bigcup\limits_{n=1}^\infty E_n\subset \Phi^{-1}\cup\{e_G\}$.

Write $\Phi^{-1}=\{g_n\}_{n=1}^\infty$.
For each $n\in\mathbb{N}$, put $F_n=E_{i_n}\cup\{g_1, \cdots, g_n\}$, where $\{i_n\}_{n\ge1}\subset\mathbb{N}$ is such that
$1<i_1<i_2<\cdots$ and $|E_{i_n}|>n^2$.
Then $$e_G\in F_1\subset F_2\subset\cdots\subset\bigcup\limits_{n=1}^\infty F_n=\Phi^{-1}\cup\{e_G\}.$$
Since $$\lim\limits_{n\rightarrow\infty}\frac{|\widetilde{g}F_n\Delta F_n|}{|F_n|}
\le\lim\limits_{n\rightarrow\infty}\frac{|\widetilde{g}E_{i_n}\Delta E_{i_n}|+2n}{|F_n|}\le
\lim\limits_{n\rightarrow\infty}(\frac{|\widetilde{g}E_{i_n}\Delta E_{i_n}|}{|E_{i_n}|}+\frac{2}{n})=0$$
for every $\widetilde{g}\in G$,
$\{F_n\}_{n=1}^\infty$ is a F\o lner sequence of $G$.
Finally, we have $$\sharp\{s\in \bigcup\limits_{n=1}^\infty F_n :s<_\Phi g^{-1}\}=|g(\Phi^{-1}\cup\{e_G\})\cap \Phi|=\infty.$$
Thus, such $g^{-1}$ and $\{F_n\}_{n\ge1}$ are as required.

(2)$\Rightarrow$(1): Suppose $g\in G$, $\{F_n\}_{n=1}^\infty$ is a F\o lner sequence of $G$ with
$\bigcup\limits_{n=1}^\infty F_n\subset \Phi^{-1}\cup\{e_G\}$, and
 $\sharp\{s\in \bigcup\limits_{n=1}^\infty F_n :s<_\Phi g\}=\infty$.
Then $\sharp\{s\in \Phi^{-1}\cup\{e_G\} :s<_\Phi g\}=\infty$.
Thus,
$$|g^{-1}\Phi^{-1}\cap\Phi|
=| \Phi^{-1} \cap g\Phi |
=\sharp\{s\in \Phi^{-1} :s\in g\Phi \}
=\sharp\{s\in \Phi^{-1} :s<_\Phi g\}=\infty,$$
which shows that $g^{-1}$ is as required.
\end{proof}

\medskip

\section{Proof of Theorem \ref{thm1}}

To prepare for the proof of Theorem \ref{thm1},
we firstly recall some useful lemmas.
Throughout this section, we let $G$ be a countable discrete infinite bi-orderable amenable group with the algebraic past $\Phi$.
Let $\{g_i\}_{i\geq 1}$ be a sequence in $G$. We say that
 $\{g_i\}_{i\geq1}$ {\it increasingly goes to infinity with
respect to $\Phi$} (write $g_i\nearrow \infty$ w.r.t. $\Phi$) if
$g_i<_\Phi g_{i+1}$ for each $i\ge 1$ and for each element $g\in G$,
$\#\{ i\in \mathbb{N}: g_i<_\Phi g\}<+\infty$. Similarly, we say
that the sequence $\{g_i\}_{i\geq1}$ {\it decreasingly goes to infinity
with respect to $\Phi$} (write $g_i\searrow \infty$ w.r.t. $\Phi$) if
$g_i>_\Phi g_{i+1}$ for each $i\ge 1$ and  for each element $g\in
G$, $\#\{ i\in \mathbb{N}: g_i>_\Phi g\}<+\infty$.
We have the following version of Pinsker formula (see \cite{AAHXY} for details).

\begin{Lemma}\label{lm1} (Pinsker formula) Let $(X,G)$ be a $G$-system,
$\mu\in \mathcal{M}(X,G)$, and $\mathcal{A}$ be a $G$-invariant
sub-$\sigma$-algebra of $\mathcal{B}_X^\mu$.  If
$\alpha,\beta,\gamma\in \mathcal{P}_X^\mu$,  and
$f_n\nearrow \infty$ w.r.t.  $\Phi$ with $f_n\Phi  f_n^{-1}=\Phi$ for each $n\ge 1$,
then the following holds:
\begin{itemize}
\item[{\rm 1)}] $h_\mu(G,\alpha\vee\beta|\mathcal{A})=h_\mu(G,\beta|\mathcal{A})+H_\mu(\alpha|\beta_G\vee\alpha_\Phi\vee \mathcal{A})$,
where $\beta_G=\bigvee\limits_{g\in G}g\beta$ and
$\alpha_\Phi=\bigvee\limits_{g\in \Phi}g\alpha$;  in particular, $h_\mu(G,\alpha|\mathcal{A})=H_\mu(\alpha|\alpha_\Phi \vee \mathcal{A})$;
\item[{\rm 2)}]  if in addition, $\alpha\preceq\beta$,  then
$\lim \limits_{n\rightarrow\infty} H_\mu\big(\alpha|\beta_\Phi\vee (f_n^{-1}\gamma)_\Phi\vee \mathcal{A}\big)=H_\mu(\alpha|\beta_\Phi\vee \mathcal{A})$.
\end{itemize}
\end{Lemma}

Next we recall \cite[Lemma 2.1]{HLY} which will be used in the proof of our main result.

\begin{Lemma}\label{l021}
Let $X$ be a compact metric space, and $\mu\in\mathcal{M}(X)$.  If $\mathcal{F}_1$ and $\mathcal{F}_2$ are two sub-$\sigma$-algebras of $\mathcal{B}_X^\mu$ with $\mathcal{F}_2\subset\mathcal{F}_1$,  and  $\mu=\int_X\mu_x^i  d\mu(x)$  is the disintegration of $\mu$ over $\mathcal{F}_i$ for $i=1,2$,  then $\text{supp}(\mu_x^1)\subset\text{supp}(\mu_x^2)$  for $\mu$-a.e.  $x\in X$.
\end{Lemma}

We also need the following result due to Mycielski  (see e.g., \cite[Theorem 5.10]{Ak}).

\begin{Lemma}\label{l023}  (Mycielski's Lemma)
Let $Y$ be a perfect compact metric space and $C$ be a symmetric dense $G_\delta$ subset of  $Y\times Y$. Then there exists a dense subset $K\subset Y$ which is a union of countably many Cantor sets such that $K\times K \subset C\cup  \Delta_Y$, where $\Delta_Y=\{(y,y):y\in Y\}$.
\end{Lemma}

\medskip

\begin{proof}[Proof of Theorem \ref{thm1}]
Denote by $\rho$ the metric on $X$.  Firstly,
 let $P_\mu(G)$ be the Pinsker $\sigma$-algebra of the system $(X,\mathcal{B}_X^\mu, G, \mu)$, i.e.,
$P_\mu(G)=\big\{ A\in \mathcal{B}_X^\mu:
h_\mu(G,\{A,X\setminus A\})=0 \big\}$.
It follows from Lemma \ref{lm1} that
$P_\mu(G)$ is a $G$-invariant
sub-$\sigma$-algebra of $\mathcal{B}_X^\mu$.

The proof will be divided into four steps.

\medskip

\noindent{\bf Step 1.  }
Since for any finitely many elements of $\Phi$ there exists an element of $\Phi$ smaller than all these elements (for instance their product),
one can easily select a sequence $h_n\searrow\infty$ in $\Phi$. Setting $f_n=h_n^{-1}$ we get a sequence $f_n\nearrow\infty$ in $\Phi^{-1}$.

\medskip

\noindent{\bf Step 2.  }
We will construct a particular measurable partition of the system $(X,\mathcal{B}_X^\mu, G, \mu)$ which plays an important role in our proof.
Let $\{\alpha_n\}_{n\geq 1}$ be an increasing sequence of finite Borel partitions of $X$ with $\text{diam}(\alpha_n)\rightarrow 0$ as $n\rightarrow \infty$.
Applying Lemma \ref{lm1} inductively, we can find a sequence $\{k_i\}_{i=1}^\infty$ such that $k_i\nearrow \infty$ and
for each $q\geq 2$,
\begin{align}\label{bound}
H_\mu\big(P_j|(P_{q-1})_\Phi\vee P_\mu(G)\big)-H_\mu\big(P_j|(P_q)_\Phi\vee
P_\mu(G)\big)<\frac{1}{j2^{q-j}},\ j=1,2,\cdots,q-1,
\end{align}
 where $P_j=\bigvee\limits_{i=1}^j f_{k_i}^{-1}\alpha_i$.
Then by the same argument as in the proof of Lemma 4.4 (1)--(2) in \cite{AAHXY}, it follows that the measurable partition
 $\mathcal{P}=\bigvee\limits_{j=1}^\infty P_j$ of $(X,\mathcal{B}_X^\mu, G, \mu)$ satisfies the following two properties:

\medskip

\noindent{\bf Claim 1.}  $\overline{\big(g\mathcal{P}_\Phi\big)(x)}\subset W_S(x,G)$ for any $x\in X$ and $g\in G$.

\medskip

\noindent{\bf Claim 2.}  $\bigcap\limits_{n=1}^\infty h_n\big(\widehat{\mathcal{P}_\Phi}\vee P_\mu(G)\big)=P_\mu(G)$.

\medskip

\noindent{\bf Step 3.  }
Let $\mu=\int_X\mu_x d\mu(x)$ be the disintegration of $\mu$ over $P_\mu(G)$.
We will show the following

\medskip

\noindent{\bf Claim 3.} There exists a set $X_1\in\mathcal{B}_X^\mu$ with $\mu(X_1)=1$, such that for any $x\in X_1$,
 $$\text{supp}(\mu_x)=\overline{\text{supp}(\mu_x)\cap W_S(x,G)}. $$

\medskip

For every $n\geq 1$, let $\mu=\int_X\mu_x^n d\mu(x)$ be the disintegration of $\mu$ over $\widehat{h_n\mathcal{P}_\Phi}\vee P_\mu(G)$.
Take a measurable partition $\xi$ such that $\widehat{\xi}=P_\mu(G)$, then we have
$\widehat{h_n\mathcal{P}_\Phi\vee\xi}=\widehat{h_n\mathcal{P}_\Phi}\vee P_\mu(G)$.
So for every $n\geq 1$, $\mu_x^n\big((h_n\mathcal{P}_\Phi\vee \xi)(x)\big)=1$,
and hence $\mu_x^n\big((h_n\mathcal{P}_\Phi)(x)\big)=1$.  Moreover,
\begin{equation}\label{wsxg1}
\mu^n_x\big(W_S(x,G)\big)=1
\end{equation}
by Claim 1.

Since $h_{n+1}<_\Phi h_n$, we have $h_n^{-1}h_{n+1}\in\Phi$, which implies $h_n^{-1}h_{n+1}\widehat{\mathcal{P}_\Phi}\subset\widehat{\mathcal{P}_\Phi}$ and thus
$\widehat{h_1\mathcal{P}_\Phi}\supset\widehat{h_2\mathcal{P}_\Phi}\supset\widehat{h_3\mathcal{P}_\Phi}\supset\cdots$. Therefore,
$$\widehat{h_1\mathcal{P}_\Phi}\vee P_\mu(G)\supset\widehat{h_2\mathcal{P}_\Phi}\vee P_\mu(G)\supset\cdots\supset\bigcap_{n=1}^\infty h_n\big(\widehat{\mathcal{P}_\Phi}\vee P_\mu(G)\big)=P_\mu(G)$$  by Claim 2.
Hence for $\mu$-a.e. $x\in X$,
\begin{equation}\label{supppp}
\text{supp}(\mu_x^n)\subset\text{supp}(\mu_x^{n+1})\subset\text{supp}(\mu_x)
\end{equation}
by Lemma \ref{l021}.

For $\mu$-a.e. $x\in X$ and any $f\in C(X)$, it follows from Martingale Theorem that
$$\int_X f d\mu_x^n = \mathbb{E}_{\mu}\big(f|\widehat{h_n\mathcal{P}_\Phi}\vee P_\mu(G)\big)(x)
\rightarrow \mathbb{E}_{\mu}\big(f| P_\mu(G)\big)(x) = \int_X f d\mu_x$$
as $n\rightarrow\infty$, which implies that
\begin{equation}\label{weakstar}
\lim_{n\rightarrow\infty} \mu_x^n =\mu_x
\end{equation}
under the weak$^*$-topology for $\mu$-a.e. $x\in X$ (see for example \cite[Corollary 5.21]{AAEW}).

Now for every $x\in X$, it is clear that $\overline{\text{supp}(\mu_x)\cap W_S(x,G)}\subset\text{supp}(\mu_x)$.
Also, by \eqref{weakstar}, \eqref{supppp} and \eqref{wsxg1}, we have for $\mu$-a.e. $x\in X$,
\begin{align*}
\mu_x\big(\overline{\text{supp}(\mu_x)\cap W_S(x,G)}\big)
&\geq\limsup_{n\rightarrow\infty}\mu_x^n\big(\overline{\text{supp}(\mu_x)\cap W_S(x,G)}\big) \\
&\geq\limsup_{n\rightarrow\infty}\mu_x^n\big(\text{supp}(\mu_x^n)\cap W_S(x,G)\big) \\&= 1
\end{align*}
which implies that $\text{supp}(\mu_x)\subset\overline{\text{supp}(\mu_x)\cap W_S(x,G)}$.
This proves Claim 3.

\medskip

\noindent{\bf Step 4.  }
We will do some preparation first and then complete the proof using Mycielski's Lemma.

\medskip

Let $\lambda=\mu \times_{P_\mu(G)}\mu$ and $\Delta_X=\{(x,x):x\in X\}$.
It is well known that  $\lambda\in \mathcal{M}^e(X\times X,G)$  and
$\lambda(\Delta_X)=0$
since  $\mu\in \mathcal{M}^e(X,G)$  and $h_\mu(G)>0$   (see e.g., \cite[Lemma 4.3]{AAHXY}).

\medskip

Clearly, $$X\times X\setminus\Delta_X=\bigcup_{k=1}^\infty\Big\{(x,y)\in X\times X: \rho(x,y)>\frac{1}{k}\Big\}. $$
 Thus there exists $\tau>0$ such that $\lambda(W)>0$, where $W=\{(x,y)\in X\times X: \rho(x,y)>\tau\}$.

Let $\eta=\tau\lambda(W)>0$. Put
$$P(X,G)=\bigcap_{m=1}^\infty\bigcap_{l=1}^\infty\bigcup_{n\geq l}\Big\{(x,y)\in X\times X : \frac{1}{|F_n|}\sum_{g\in F_n}\rho(gx,gy)<\frac{1}{m}\Big\}$$  and
$$D_\eta(X,G)=\bigcap_{m=1}^\infty\bigcap_{l=1}^\infty\bigcup_{n\geq l}\Big\{(x,y)\in X\times X : \frac{1}{|F_n|}\sum_{g\in F_n}\rho(gx,gy)>\eta-\frac{1}{m}\Big\}.$$
It is easy to check that $P(X,G)$ and $D_\eta(X,G)$ are $G_\delta$ subsets of $X\times X$.
Note that
\begin{equation}\label{mdp}
MLY_\eta(X,G)=P(X,G)\cap D_\eta(X,G)
\end{equation}
which is also a $G_\delta$ subset of $X\times X$.

\medskip

Since $\{F_n\}_{n=1}^\infty$ is a F\o lner sequence of $G$, we can find a tempered F\o lner subsequence
$\{\widehat{F_n}\}_{n=1}^\infty\subset \{F_n\}_{n=1}^\infty$ (see e.g., \cite[Proposition 1.5]{Lindenstrauss});
that is, such that there exists $b>0$ satisfying
$$\bigg|\bigcup_{k<n} \widehat{F_k}^{-1}\widehat{F_n}\bigg|\leq b\big|\widehat{F_n}\big| $$
 for any  $n\geq 1$.

Let $G_\lambda$ be the set of all generic points of $\lambda$ with respect to $\{\widehat{F_n}\}_{n=1}^\infty$; that is,
$(x,y)\in G_\lambda$ if and only if
$$\frac{1}{|\widehat{F_n}|}\sum_{g\in \widehat{F_n}}\delta_{g(x,y)}\rightarrow\lambda$$ under the weak$^*$-topology, where $\delta_{g(x,y)}$ is the point mass on $g(x,y)=(gx,gy)$.

Choose a countable dense subset $\{\phi_k\}_{k=1}^\infty\subset C(X\times X)$, and for each $k\geq 1$,  put
$$G_k=\bigg\{(x,y)\in X\times X : \lim\limits_{n\rightarrow\infty}\frac{1}{|\widehat{F_n}|}\sum_{g\in \widehat{F_n}}\phi_k\big(g(x,y)\big) = \int \phi_k d \lambda\bigg\}.$$
Then it is not hard to see that
$G_\lambda=\bigcap\limits_{k=1}^\infty G_k$.
By Lindenstrauss Pointwise Ergodic Theorem (see e.g., \cite[Theorem 1.2]{Lindenstrauss}),
 we have  $\lambda(G_k)=1$ for each $k\geq 1$.
Thus $\lambda(G_\lambda)=1$.

\medskip

For $(x,y)\in G_\lambda$, noting that $W$ is open, we have
\begin{align*}
\limsup_{n\rightarrow\infty} \frac{1}{|F_n|} \sum_{g\in F_n} \rho(gx,gy) \ge
\limsup_{n\rightarrow\infty} \frac{1}{|\widehat{F_n}|} \sum_{g\in \widehat{F_n}} \rho(gx,gy)
=\int \rho(x,y) d \lambda \ge \tau\lambda(W),
\end{align*}
because $\rho(\cdot,\cdot)\in C(X\times X)$.
This implies that $G_\lambda\subset D_\eta (X,G)$.

\medskip

Since $$\int (\mu_x \times \mu_x) (\Delta_X) d \mu (x) =\lambda(\Delta_X)=0, $$
we have
$(\mu_x\times\mu_x)(\Delta_X)=0$ for $\mu$-a.e. $x\in X$. Thus,
$\mu_x$ is non-atomic for $\mu$-a.e. $x\in X$.
This, together with $$\int_X(\mu_x\times\mu_x)(G_\lambda)d\mu(x)=\lambda(G_\lambda)=1,$$ implies that,
there exists $X_2\in\mathcal{B}_X^\mu$ with $\mu(X_2)=1$ such that $\mu_x$ is non-atomic and $$(\mu_x\times\mu_x)(G_\lambda)=1$$ for all $x\in X_2$.

\medskip

Let $X_0=X_1\bigcap X_2$. Then $\mu(X_0)=1$.
Now take $x\in X_0$. Then $\mu_x$ is non-atomic and hence
\begin{equation}\label{perfect}
\text{supp}(\mu_x) \text{ is a perfect closed subset of } X.
\end{equation}
Since $x\in X_2$, $$(\mu_x\times\mu_x)\big(G_\lambda\cap(\text{supp}(\mu_x)\times\text{supp}(\mu_x))\big)=1.$$ It follows that
$$G_\lambda\cap\big(\text{supp}(\mu_x)\times\text{supp}(\mu_x)\big) \text{ is dense in }\text{supp}(\mu_x)\times\text{supp}(\mu_x)$$ and hence
\begin{equation}\label{ddd}
D_\eta(X,G)\cap\big(\text{supp}(\mu_x)\times\text{supp}(\mu_x)\big) \text{ is a dense } G_\delta \text{ subset of } \text{supp}(\mu_x)\times\text{supp}(\mu_x)
\end{equation}
by the fact $G_\lambda\subset D_\eta(X,G)$.

\medskip

Since $x\in X_1$ and $W_S(x,G)\times W_S(x,G)\subset Asy_S(X,G)$, we have
\begin{align*}
&Asy_S(X,G)\cap\big(\text{supp}(\mu_x)\times\text{supp}(\mu_x)\big) \\
\supset&\big(W_S(x,G)\times W_S(x,G)\big)\cap\big(\text{supp}(\mu_x)\times\text{supp}(\mu_x)\big) \\
\supset&\big(W_S(x,G)\cap\text{supp}(\mu_x)\big)\times\big(W_S(x,G)\cap\text{supp}(\mu_x)\big)
\end{align*}
which implies that
\begin{align*}
&\overline{Asy_S(X,G)\cap\big(\text{supp}(\mu_x)\times\text{supp}(\mu_x)\big)} \\
\supset&\overline{W_S(x,G)\cap
\text{supp}(\mu_x)}\times\overline{W_S(x,G)\cap\text{supp}(\mu_x)} \\
=&\text{supp}(\mu_x)\times\text{supp}(\mu_x).
\end{align*}
Thus,
\begin{equation}\label{302}
Asy_S(X,G)\cap\big(\text{supp}(\mu_x)\times\text{supp}(\mu_x)\big) \text{  is dense in   }  \text{supp}(\mu_x)\times\text{supp}(\mu_x).
\end{equation}

\medskip

Since $\{F_n\}_{n=1}^\infty$ is a F{\o}lner sequence of the infinite group $G$,
 we have $|F_n|\rightarrow\infty$.
Hence $Asy_S(X,G)\subset P(X,G)$ by noting that $\bigcup\limits_{n=1}^\infty F_n\subset S$ and the fact that $|F_n|\rightarrow\infty$.

Combining this with \eqref{302}, we have
\begin{equation}\label{ppp}
P(X,G)\cap\big(\text{supp}(\mu_x)\times\text{supp}(\mu_x)\big) \text{ is a dense } G_\delta \text{ subset of }  \text{supp}(\mu_x)\times\text{supp}(\mu_x).
\end{equation}
Thus,
\begin{equation}\label{dengd}
MLY_\eta(X,G)\cap\big(\text{supp}(\mu_x)\times\text{supp}(\mu_x)\big) \text{ is a dense } G_\delta \text{ subset of } \text{supp}(\mu_x)\times\text{supp}(\mu_x)
\end{equation}
by \eqref{ppp}, \eqref{ddd} and \eqref{mdp}.

\medskip

Now by \eqref{perfect}, \eqref{dengd} and  Mycielski's Lemma (Lemma \ref{l023}), there exists a dense subset $K_x$ of $\text{supp}(\mu_x)$ which is a union of countably many Cantor sets such that
$$K_x\times K_x\subset MLY_\eta(X,G)\cup\Delta_X.$$
This completes the proof.
\end{proof}

\medskip

\section{Proof of Theorem \ref{thm2}}

Firstly, we prove the following lemma.
Recall that $U_{d+1}(\mathbb{Z})$  and $\{F_n\}_{n=1}^\infty$ are defined in \eqref{udz} and \eqref{udzfn}, respectively.

\begin{Lemma}\label{fphilners}
$\{F_n\}_{n\ge 1}$ defined in \eqref{udzfn} is a F\o lner sequence of $U_{d+1}(\mathbb{Z})$.
\end{Lemma}

\begin{proof}
It is clear that for  any $A\in
U_{d+1}(\mathbb{Z})$ there exists a unique $${\bf a}=(a_{i,k})_{1\le
k\le d, 1\le i\le d-k+1}\in \mathbb{Z}^{d(d+1)/2}$$ such that
$A=M({\bf a})$.
Moreover, for given $${\bf
a}=(a_{i,k})_{1\le k\le d, 1\le i\le d-k+1}\in \mathbb{Z}^{d(d+1)/2},
  {\bf b}=(b_{i,k})_{1\le k\le d, 1\le i\le d-k+1}\in
\mathbb{Z}^{d(d+1)/2},$$   if ${\bf c}=(c_{i,k})_{1\le k\le
d, 1\le i\le d-k+1}\in \mathbb{Z}^{d(d+1)/2}$  is  such
that $$M({\bf c})=M({\bf a})M({\bf b}),$$  then
\begin{equation}\label{c-eq-1}
c_{i,k}=a_{i,k}+(\sum
\limits_{j=1}^{k-1}a_{i,k-j}b_{i+k-j,j})+b_{i,k}
\end{equation}
for $1\le k \le d$ and $1\le i \le d-k+1$.

\medskip

Fix $g=M({\bf a})\in U_{d+1}(\mathbb{Z})$.
 It suffices to show that
$$\lim_{n\rightarrow +\infty} \frac{|gF_n\Delta F_n|}{|F_n|}=0. $$

Let $D=\Big( \sum\limits_{k=1}^d k(d-k+1)\Big) -1$.
Then when $n\in\mathbb{N}$ is large enough, we have
\begin{align}\label{fenmu}
|F_n|&=\sum_{l=0}^n \bigg( \frac{1}{l+1} \prod_{k=1}^d (l^k+1)^{d-k+1} \bigg) \nonumber\ge   \sum_{l=1}^n  l^{\big( \sum\limits_{k=1}^d k(d-k+1)\big) -1} \nonumber\\
&=\sum_{l=1}^n l^D \ge\sum_{l=[\frac{n+1}{2}]}^n l^D \ge (\frac{n-1}{2})^{D+1} .
\end{align}
Note that to figure out $|F_n|$, we let $l$ represent the value of $b_{d,1}$, and should divide by $l+1$, because
the next product pretends that there are $l+1$ choices for $b_{d,1}$, while it is fixed; then we get the count
$\sum_{l=0}^n \big( \frac{1}{l+1} \prod_{k=1}^d (l^k+1)^{d-k+1} \big)$.

\medskip

Next we show that for this given $g=M({\bf a})$, there exists
$L=L(g)>0$ such that $$|gF_n\setminus F_n|\le Ln^D$$ when $n$ is large enough.

By \eqref{c-eq-1}, it is not hard to check that
\begin{align*}
 &\left| gF_n\setminus F_n \right|=\left| F_n\setminus g^{-1}F_n \right|\\
=&\sharp\left\{M({\bf b})\in U_{d+1}(\mathbb{Z}) :  M({\bf b})\in F_n , M({\bf a})M({\bf b})\notin F_n \right\} \\
=&\sharp\left\{{\bf b}\in \mathbf{S} :  b_{d,1} \le n, M({\bf a})M({\bf b})\notin F_n  \right\}\\
\leq & \sharp\left\{{\bf b}\in \mathbf{S} :  b_{d,1}\le n,  a_{d,1}+b_{d,1}>n \right\} \\
& + \sum_{P=1}^d\sum_{Q=1}^{d-P+1} \sharp\left\{{\bf b}\in \mathbf{S} :   b_{d,1}\le n,
(a_{d,1}+b_{d,1})^P < a_{Q,P}+(\sum\limits_{j=1}^{P-1}a_{Q,P-j}b_{Q+P-j,j})+b_{Q,P}  \right\} \\
& + \sum_{P=1}^d\sum_{Q=1}^{d-P+1} \sharp\left\{{\bf b}\in \mathbf{S} :  b_{d,1}\le n,
a_{Q,P}+(\sum\limits_{j=1}^{P-1}a_{Q,P-j}b_{Q+P-j,j})+b_{Q,P} < 0  \right\} .
\end{align*}

Let $R=\max \big\{ \big| a_{i,k} \big| : 1\le k\le d, 1\le i\le d-k+1 \big\}$, and $H=a_{d,1}$.
We now estimate the above three items respectively.
For the first item, noting that
$$ \frac{1}{l+1} \prod_{k=1}^d (l^k+1)^{d-k+1}
= (l+1)^{d-1} \prod_{k=2}^d (l^k+1)^{d-k+1} $$
which is increasing with respect to $l$,
we have
\begin{align*}
 & \sharp\left\{{\bf b}\in \mathbf{S} :   b_{d,1}\le n,  a_{d,1}+b_{d,1}>n \right\}
\leq  \sharp\left\{{\bf b}\in \mathbf{S} :   n-|H|+1 \le b_{d,1} \le n  \right\} \\
= & \sum_{l=n-|H|+1}^n \bigg( \frac{1}{l+1} \prod_{k=1}^d (l^k+1)^{d-k+1} \bigg)
\leq  |H| \bigg( \frac{1}{n+1} \prod_{k=1}^d (n^k+1)^{d-k+1} \bigg) \\
\leq & |H| \bigg( \frac{1}{n} \prod_{k=1}^d (2n^k)^{d-k+1} \bigg)
\leq  2^{d^2} |H| n^D .
\end{align*}

For the second item, fix $1\le P \le d$ and $1\le Q\le d-P+1$, we have
\begin{align*}
&   \sharp\left\{{\bf b}\in  \mathbf{S} :   b_{d,1}\le n, (a_{d,1}+b_{d,1})^P < a_{Q,P}+(\sum\limits_{j=1}^{P-1}a_{Q,P-j}b_{Q+P-j,j})+b_{Q,P} \right\} \\
\leq &   \sharp\left\{{\bf b}\in \mathbf{S} :  b_{d,1}\le n,  (H+b_{d,1})^P \le R+\sum\limits_{j=1}^{P-1}Rb_{d,1}^j+b_{Q,P}  \right\} \\
\leq &   \sharp\left\{{\bf b}\in \mathbf{S} :  b_{d,1}\le n,  (H+b_{d,1})^P \le R+R(P-1)b_{d,1}^{P-1}+b_{Q,P}  \right\} \\
\leq &   \sharp\left\{{\bf b}\in \mathbf{S} :  b_{d,1}\le n,  (H+b_{d,1})^P -  \big( R+Rdb_{d,1}^{P-1} \big) \le b_{Q,P} \le b_{d,1}^P   \right\} \\
\leq &   \sum_{l=0}^n  \bigg| l^P -(H+l)^P +  \big( R+Rdl^{P-1} \big) +1 \bigg| \bigg( \frac{1}{l^P+1} \bigg) \bigg( \frac{1}{l+1} \prod_{k=1}^d (l^k+1)^{d-k+1} \bigg) .
\end{align*}
Note that the factor
$\Big| l^P -(H+l)^P +  \big( R+Rdl^{P-1} \big) +1 \Big| \Big( \frac{1}{l^P+1} \Big)$
in the last inequality
results from changing the range of $b_{Q,P}$.
Since
$$ \bigg( \frac{1}{l^P+1} \bigg) \bigg( \frac{1}{l+1} \prod_{k=1}^d (l^k+1)^{d-k+1} \bigg)
= \Big( l^P+1 \Big)^{d-P} \Big(l+1\Big)^{d-1} \bigg( \prod_{2\le k\le d, k\neq P} (l^k+1)^{d-k+1} \bigg) $$
which increases with respect to $l$, and there exists $B=B({\bf a})>0$ such that
$$\bigg| n^P -(H+n)^P +  \big( R+Rdn^{P-1} \big) +1 \bigg|\le B n^{P-1}$$ when $n$ is large enough,
 we then have
\begin{align*}
&   \sharp\left\{{\bf b}\in  \mathbf{S} :   b_{d,1}\le n, (a_{d,1}+b_{d,1})^P < a_{Q,P}+(\sum\limits_{j=1}^{P-1}a_{Q,P-j}b_{Q+P-j,j})+b_{Q,P} \right\} \\
\leq & (n+1) \bigg| n^P -(H+n)^P + \big( R+Rdn^{P-1} \big) +1 \bigg| \bigg( \frac{1}{n^P+1} \bigg) \bigg( \frac{1}{n+1} \prod_{k=1}^d (n^k+1)^{d-k+1} \bigg)\\
\leq &  (n+1)  B n^{P-1}  \bigg( \frac{1}{n^P+1} \bigg) \bigg( \frac{1}{n+1} \prod_{k=1}^d (n^k+1)^{d-k+1} \bigg)\\
\leq &  B n^{P-1}  \bigg( \frac{1}{n^P} \bigg) \bigg(  \prod_{k=1}^d (2n^k)^{d-k+1} \bigg) \\
\leq &  2^{d^2} B n^D .
\end{align*}

For the third item, fix $1\le P \le d$ and $1\le Q\le d-P+1$. Similarly, when $n$ is large enough, we have
\begin{align*}
&   \sharp\left\{{\bf b}\in \mathbf{S} : b_{d,1}\le n,  a_{Q,P}+(\sum\limits_{j=1}^{P-1}a_{Q,P-j}b_{Q+P-j,j})+b_{Q,P} < 0  \right\} \\
\leq &   \sharp\left\{{\bf b}\in \mathbf{S} :  b_{d,1}\le n,  b_{Q,P} \le R+\sum\limits_{j=1}^{P-1}Rb_{d,1}^j  \right\} \\
\leq &   \sharp\left\{{\bf b}\in \mathbf{S} :  b_{d,1}\le n,  b_{Q,P} \le R+R(P-1)b_{d,1}^{P-1}  \right\} \\
\leq &   \sharp\left\{{\bf b}\in \mathbf{S} :  b_{d,1}\le n,   0 \le  b_{Q,P} \le R+Rdb_{d,1}^{P-1}  \right\} \\
= &   \sum_{l=0}^n  \bigg( R+Rdl^{P-1} +1 \bigg) \bigg( \frac{1}{l^P+1} \bigg) \bigg( \frac{1}{l+1} \prod_{k=1}^d (l^k+1)^{d-k+1} \bigg) \\
\leq &   (n+1)  \bigg( R+Rdn^{P-1} +1 \bigg) \bigg( \frac{1}{n^P+1} \bigg) \bigg( \frac{1}{n+1} \prod_{k=1}^d (n^k+1)^{d-k+1} \bigg)\\
\leq &  \bigg( (2Rd+1)n^{P-1}  \bigg) \bigg( \frac{1}{n^P+1} \bigg) \bigg(  \prod_{k=1}^d (n^k+1)^{d-k+1} \bigg)\\
\leq &  \bigg( (2Rd+1)n^{P-1}  \bigg) \bigg( \frac{1}{n^P} \bigg) \bigg(  \prod_{k=1}^d (2n^k)^{d-k+1} \bigg)\\
\leq &  2^{d^2} C n^D ,
\end{align*}
where $C=C({\bf a})=2Rd+1>0$.

Summarizing up, when $n\in \mathbb{N}$ is large enough, we have
\begin{align}\label{fenzi}
\big|gF_n\setminus F_n\big|
\le  \big( 2^{d^2} |H| n^D \big) + \sum_{P=1}^d\sum_{Q=1}^{d-P+1} \Big( 2^{d^2} B n^D + 2^{d^2} C n^D \Big)
= Ln^D  ,
\end{align}
where $L= 2^{d^2} |H|  + \sum\limits_{P=1}^d\sum\limits_{Q=1}^{d-P+1} \Big( 2^{d^2} B + 2^{d^2} C \Big)$ is a positive constant.

\medskip

Therefore, by \eqref{fenmu} and \eqref{fenzi}, we have
$$\frac{|gF_n\setminus F_n|}{|F_n|}\leq\frac{Ln^D}{ (\frac{n-1}{2})^{D+1} } \rightarrow 0$$
as $n\rightarrow\infty$. It follows that for every
$g \in U_{d+1}(\mathbb{Z})$,
\begin{equation*}
\frac{\big|gF_n\Delta F_n\big|}{\big|F_n\big|}
=\frac{\big|gF_n\setminus F_n\big|+\big|F_n\setminus gF_n\big|}{\big|F_n\big|}
=\frac{2 \big|gF_n\setminus F_n\big|}{\big|F_n\big|}
   \rightarrow 0
\end{equation*}
as $n\rightarrow\infty$,
which implies that such  $\big\{F_n\big\}_{n\ge 1}$ is a F\o lner sequence of $U_{d+1}(\mathbb{Z})$.
\end{proof}

\medskip

Finally, we prove Theorem \ref{thm2}. Recall that $S$ is defined in \eqref{udzs}.

\begin{proof} [Proof of Theorem \ref{thm2}]
Define the following linear order relation on
$U_{d+1}(\mathbb{Z})$:  $$M({\bf a})<M({\bf b})$$  if and only if
$${\bf a}=\big(a_{d,1},a_{d-1,1},\cdots,a_{1,1};a_{d-1,2},a_{d-2,2},\cdots,a_{1,2};a_{d-2,3},\cdots,a_{1,3};\cdots;a_{2,d-1},a_{1,d-1};a_{1,d}\big)$$
is lexicographically less than
$${\bf b}=\big(b_{d,1},b_{d-1,1},\cdots,b_{1,1};b_{d-1,2},b_{d-2,2},\cdots,b_{1,2};b_{d-2,3},\cdots,b_{1,3};\cdots;b_{2,d-1},b_{1,d-1};b_{1,d}\big).$$

This order relation is right-invariant (also left-invariant) with respect to the translations of $U_{d+1}(\mathbb{Z})$,
which is not hard to see from \eqref{c-eq-1};
when computing the term $c_{i,k}$ of the matrix $M({\bf a})M({\bf b})$ we only use these terms of
$M({\bf a})$ which appear to the left of $a_{i,k}$ in the above writing of ${\bf a}$ as a vector.

Thus,
we obtain an algebraic past $\Phi$ in
 $U_{d+1}(\mathbb{Z})$ defined as a subset of all elements of
$U_{d+1}(\mathbb{Z})$ which are less than the identity $I_{d+1}$;
 that is, $\Phi=\{M({\bf a})\in U_{d+1}(\mathbb{Z}): M({\bf a})<I_{d+1}\}$.
By Lemma \ref{fphilners}, we know that $\{F_n\}_{n\ge 1}$ is a F\o lner sequence of $U_{d+1}(\mathbb{Z})$.
Thus, $U_{d+1}(\mathbb{Z})$ is a countable discrete
infinite bi-orderable amenable group with the algebraic past $\Phi$.

It is clear that $S$ is infinite and
$I_{d+1}\in F_1\subset F_2\subset\cdots\subset\bigcup_{n=1}^\infty F_n\subset S$.
For each $g\in U_{d+1}(\mathbb{Z})$,
we can easily check that
$\sharp\big\{s\in S:s<_\Phi
g\big\}<\infty$.
Hence the result follows from Theorem \ref{thm1}.
\end{proof}

\end{document}